\documentclass[12pt]{article}
\usepackage{amssymb,amsmath}

\def\G{\Gamma}
\def\a{\alpha}
\def\b{\beta}
\def\o{\omega}
\def\<{\langle}
\def\>{\rangle}
\def\Q{\mathbb{Q}}
\def\Z{\mathbb{Z}}
\def\sm{\smallsetminus}

\newtheorem{Theorem}{Theorem}[section]
\newtheorem{Proposition}[Theorem]{Proposition}

\newenvironment{proof}{\par\medskip\noindent{\it Proof.}~}{\hfill $\square$}

\title{Freiheitss\"{a}tze for one-relator quotients of surface groups
and of limit groups
}

\author{James Howie and Muhammad Sarwar Saeed\\
Department of Mathematics and\\ Maxwell Institute of Mathematical Sciences\\
Heriot-Watt University\\ Edinburgh EH14 4AS\\ UK}

\begin{document}

\maketitle

\begin{abstract}
Three versions of the Freiheitssatz are proved in the context of one-relator quotients of limit groups, where the latter are equipped with $1$-acylindrical splittings over cyclic subgroups.  These are natural extensions of previously published corresponding statements for one-relator quotients of orientable surface groups.  Two of the proofs are new even in that restricted context.
\end{abstract}

\section{Introduction}

The rich theory of groups with a single defining relator has inspired a number of generalizations, in which a one-relator group $F/N(R)$ (where $F$ is a free group, $R$ an element of $F$, and $N(R)$ its normal closure) is replaced by a group of the form $L/N(R)$ for some other free construction $L$.   For example, there is now a fairly well-developed theory of one-relator quotients of free products (see for example the survey article \cite{alta}), and also some work on one-relator quotients of free products with amalgamation and of HNN extensions \cite{FRR, Juh1, Juh2, Juh3}. In another direction, there are the beginnings of a theory of one-relator quotients of (orientable) surface groups \cite{Bog,Hem90,How,Err,MSS}.

Recent major advances in the logical theory of free groups (see \cite{KM,Sela6} and the references cited therein), surrounding the Tarski problems, has brought 
much attention to bear on another class of groups that generalize free groups, namely the {\em limit groups} (also variously known as {\em fully residually free groups} or {\em $\omega$-residually free groups}).  It is therefore natural to consider how to develop a theory of one-relator quotients of limit groups.

A fundamental result which is crucial to the development of any one-relator theory is the Freiheitssatz of Magnus \cite{Mag} and its variants.  In its original formulation, this says that if $Y$ is a subset of a basis $X$ for the free group $F=F(X)$, such that $R$ is not conjugate to a word in $Y$, then $Y$ freely generates a subgroup of the one-relator group $F(X)/N(R)$.  Equivalently, the natural map
$$F(Y) \to F(X) \to F(X)/N(R)$$
is injective.

In order to formulate a Freiheitssatz for a generalized one-relator theory, one first has to identify suitable analogues for these so-called {\em Magnus subgroups} $F(Y)$.  For one-relator quotients of free products $(A*B)/N(R)$ the obvious candidates are the free factors $A,B$.  The Freiheitssatz does not hold in general for these groups: there are examples where the natural map $A\to A*B\to (A*B)/N(R)$ is not injective.  Nevertheless, there are various results giving sufficient conditions for a Freiheitssatz of this form to hold, and 
it is only under such conditions that a further development of the theory
has proved possible.

For a one-relator surface group $\pi_1(\Sigma)/N(R)$, the natural candidates for Magnus subgroups are the fundamental groups $\pi_1(\Sigma_0)$ of incompressible subsurfaces of $\Sigma$, but again it is not always true that the Freiheitssatz hold for all such subgroups.  A very general form of the Freiheitssatz was wrongly asserted in \cite{How}, but the error pointed out and counterexamples
given in \cite{Err}. Nevertheless, one correct Freiheitssatz was proved in \cite{How}:

\begin{Theorem}\label{h0}{\rm \cite[Proposition 3.10]{How}}
Let $S$ be a closed oriented surface, $\a$ a closed curve in $S$, and
$\b$ a simple closed curve in $S$ such that $\a$ is not
homotopic to a curve disjoint from $\b$, and that $\<\a,\b\>=0$.  Then
$\pi_1(S\smallsetminus\b)\to\pi_1(S)/N(\a)$ is injective.
\end{Theorem}

(Here $\<\a,\b\>$ denotes the algebraic intersection number of the
two curves $\a,\b$ on $S$.)

Two further versions of the Freiheitssatz for one-relator quotients of surface groups were stated (without proof) in \cite{Err}: 

\begin{Theorem}\label{h1}{\rm \cite[Theorem 6]{Err}}
Let $S$ be a closed oriented surface, $\a$ a closed curve in $S$, and
$\b$ a simple closed curve in $S$ such that $\a$ is not homotopic to a curve that meets $\b$ at most once, and that $\<\a,\b\>=\pm 1$.  Then $\pi_1(S\smallsetminus \b )\to \pi_1(S)/N(\a )$ is injective.
\end{Theorem}

\begin{Theorem}\label{h2}{\rm \cite[Theorem 7]{Err}}
Let $S$ be a closed oriented surface, $\a$ a closed curve in $S$, and
$\b_1,\b_2$ two disjoint simple closed curves in $S$ such that $\a$ is not homotopic to a curve disjoint from $\b_1$ or from $\b_2$. Then $\pi_1(S\smallsetminus (\b_1\cup\b_2) )\to \pi_1(S)/N(\a )$ is injective.
\end{Theorem}

The example of surface groups forms a useful template for a theory of one-relator quotients of limit groups.  Every nonabelian limit group splits over a cyclic (possibly trivial) subgroup.  In the case of a surface group this corresponds to splitting the surface over an essential simple closed curve, and the fundamental group of the complement of this curve is one of our candidates for Magnus subgroup. The analogue for a limit group would be a vertex group
for a cyclic splitting.  As with surface groups, the obvious analogue of the Freiheitssatz is not true in complete generality: there are examples where the vertex group does not embed into the one-relator quotient. Nevertheless, a Freiheitssatz holds in many cases, and in this paper we will prove natural analogues of Theorems \ref{h0}, \ref{h1} and \ref{h2} as follows. In these statements $\ell_e(R)$ means the $e$-length of $R$, which is the translation length arising from the splitting along the edge $e$, while $\sigma_e(R)$ means the exponent sum of $e$ in $R$. In the case of a surface group split along a simple closed curve $\b$, these correspond to the geometric intersection number 
and the algebraic intersection number, respectively, of a curve $\a$
representing $R$ with the simple closed curve $\b$.

\medskip\noindent{\bf Theorem \ref{expsumzero}}
{\it Let $\G$ be a limit group expressed as the fundamental group of a $1$-acylindrical graph of groups $(\mathcal{G},X)$, let $e$ be an edge of $X$ with cyclic edge group, and let $R\in\G$ be an element such that $\ell_e(R)>0=\sigma_e(R)$.  Let $A=\pi_1(\mathcal{G},Y)$, where $Y$ is one component of $X\sm\{e\}$. Then the natural homomorphism $A\to\G/N(R)$ is injective.}

\medskip\noindent{\bf Theorem \ref{expsumone}}
{\it Let $\G$ be a limit group expressed as the fundamental group of a $1$-acylindrical graph of groups $(\mathcal{G},X)$, let $e$ be an edge of $X$ with cyclic edge group, and let $R\in\G$ be an element such that $\ell_e(R)>1=\sigma_e(R)$.  Let $A=\pi_1(\mathcal{G},Y)$, where $Y=X\sm\{e\}$.  Then the natural homomorphism $A\to\G/N(R)$ is injective.}

\medskip\noindent{\bf Theorem \ref{expsumtwo}}
{\it Let $\G$ be a limit group expressed as the fundamental group of a $1$-acylindrical graph of groups $(\mathcal{G},X)$, let $e,f$ be edges of $X$ with cyclic edge groups, and let $R\in\G$ be an element such that $\ell_e(R)\ge 1\le\ell_f(R)$.  Let $A=\pi_1(\mathcal{G},Y)$, where $Y$ is one component of $X\sm\{e,f\}$.  Then the natural homomorphism $A\to\G/N(R)$ is injective.}

Our paper therefore serves a dual purpose.  It provides the proofs of Theorems \ref{h1} and \ref{h2} promised in \cite{Err}, while at the same time generalizing them to the context of limit groups and thereby laying the foundations for a theory of one-relator quotients of limit groups.

The rest of the paper is organised as follows: In \S\ref{limgps} we recall the basic properties of limit groups that we need for our purposes, and prove a simple version of the Freiheitssatz where the edge group is trivial (Proposition \ref{freeproduct} below).  Then \S\S \ref{es0}, \ref{es1} and \ref{2e} are devoted to the proofs of Theorems \ref{expsumzero}, \ref{expsumone} and \ref{expsumtwo} respectively.

\section{Limit groups} \label{limgps}

A {\em limit group} is a finitely generated group $G$ with the property that, for each finite subset $A\subset G$, there is a homomorphism $\phi:G\to F$
from $G$ to a free group $F$ that restricts to an injection on $A$.

Examples of limit groups include: free groups and free abelian groups of finite rank, orientable surface groups, and any free product of finitely many of the above.  More complicated examples can be constructed using Sela's {\em $\o$-residually free towers} (see \cite{Sela, BF}). These are graph-of-groups constructions with cyclic edge groups and (simpler) limit groups as vertex groups.

Conversely, any limit group which is not free abelian has a graph-of-groups
decomposition of this type.  For example, if $\Sigma$ is a closed, orientable, hyperbolic surface, then cutting $\Sigma$ along any essential simple closed curve yields a decomposition of $\pi_1(\Sigma)$ as a free product of two free groups with cyclic amalgamated subgroup, or alternatively as an HNN extension of a free base group with cyclic associated subgroups.

A one-relator quotient of a limit group is just the quotient $G=\G/N(R)$ of a limit group $\G$ by the normal closure $N(R)$ of a single element $R\in\G$. Our approach to one-relator theory on such an object will be to consider the length of $R$ in terms of a suitable graph-of-groups decomposition of $\G$.  Indeed, our approach will work equally well in circumstances where the vertex groups are limit groups, but the graph product $\G$ is not necessarily a limit group.  On the other hand, there are simple examples of limit groups in which the Freiheitssatz fails (see \cite{Err} for some surface group examples, as well as the example below). To avoid such pathological examples, we shall impose further restrictions on the graphs of groups that arise.

\medskip\noindent{\bf Example}
Let $F$ be a nonabelian free group and $w\in F$ a word which is not a proper
power in $F$. Let $\G=F\ast_C A$ be the free product of $F$ with a free abelian group $A$, amalgamating a cyclic subgroup $C$, where $C<F$ is generated by $w$, and $C<A$ is a direct factor.  Then $\G$ is a limit group.

Choose $u\in F\sm C$ and $v\in A\sm C$.  Then the natural map $F\to \G/N(uv)$
is not injective, since the commutator $[u,w]$ lies in its kernel.  

Recall that a group action on a tree is {\em $k$-acylindrical} if every geodesic segment of length greater than $k$ has trivial stabilizer.  A graph of groups is {\em $k$-acylindrical} if the action of its fundamental group on the corresponding Bass-Serre tree is $k$-acylindrical.  The natural graph-of-groups decompositions of limit groups described above are always $2$-acylindrical, but not in general $1$-acylindrical 

\medskip 
Clearly the graph of groups in the example above is not $1$-acylindrical.  We shall concentrate on $1$-acylindrical actions of limit groups on trees, whose edge-stabilisers are cyclic.

\bigskip
Note that, if $\G$ is a limit group acting $1$-acylindrically on a tree $T$, $H$ is a subgroup, $T'$ an $H$-invariant subtree of $T$, and $\bar{T}$ the quotient tree of $T'$ obtained by shrinking each component of some $H$-invariant forest to a point, then the resulting action of $H$ on $\bar{T}$ is also $1$-acylindrical. 

\medskip
Suppose that a limit group $\G$ acts $1$-acylindrically on a tree $T$ with
cyclic edge stabilisers.  Suppose that $R\in\G$ is a hyperbolic element - that
is, $R$ does not fix any vertex of $T$.  Then there is a unique {\em axis} of $R$ in $T$, in other words, an $R$-invariant bi-infinite geodesic, upon which $R$ acts by translation.  The {\em translation-length} of $R$ is the edge-path length by which $R$ shifts its axis $A$, or equivalently the number of edges in the quotient graph $\<R\>\backslash A$.

Let $e$ be an edge of $T$.  Then the $e$-length of $R$ is the number $\ell_e(R)$ of edges in $\<R\>\backslash A$ that are images of $\G$-translates of $e$. Equivalently, if $v$ is a vertex of $A$, then the geodesic $[v,R(v)]$ from $v$ to $R(v)$ is a fundamental domain for the action of $R$ on $A$, and the number of edges from the $\G$-orbit $\G e$ in $[v,R(v)]$ is $\ell_e(R)$.

Let $e$ be an oriented edge of $T$.  Then the {\em exponent-sum} of $e$ in $R$ is the number $\sigma_e(R)$ of edges in $\<R\>\backslash A$ that are images of $\G$-translates of $e$, {\em counted with multiplicity}.  Equivalently, if $v$ is a vertex of $A$, then $\sigma_e(R)$ is the number of oriented edges from $\G e$ in the oriented geodesic $[v,R(v)]$, minus the number of oriented edges from $\G e$ in the oriented geodesic $[R(v),v]$. Note that $\sigma_e(R)\equiv\ell_e(R)~\mathrm{mod}~2$.

In the particular case where $\G$ is the group of an orientable hyperbolic surface, with the tree action arising from a splitting of the surface along
an essential closed curve $\a$, then $\ell_e(R)$ is the minimal intersection number of $\a$ with any closed curve representing a conjugate of $R$, while $\sigma_e(R)$ is the algebraic intersection number of $\a$ with (any closed curve representing) $R$.

\medskip
We first note an easy special case of the Freiheitssatz.

\begin{Proposition}\label{freeproduct}
Let $\G=\pi_1(\mathcal{G},X)$ be a limit group expressed as the fundamental group of a graph of groups.  Let $e$ be an edge of $X$ such that the corresponding edge group is trivial, and let $A=\pi_1(\mathcal{G},Y)$ where $Y$ is one component of $X\sm\{e\}$.  Let $R\in\G$ be any element involving $e$, in the sense that $\ell_e(R)>0$.  Then the natural map $A\to \G/N(R)$ is injective.
\end{Proposition}

\begin{proof}
If $e$ disconnects $X$, then $\G=A*B$, where $B=\pi_1(\mathcal{G},Z)$ with $X\sm\{e\}=Y\sqcup Z$.  Now $A$ and $B$ are residually free.  Moreover by hypothesis $R\in A*B$ is not conjugate to an element of $A$ or of $B$. The result follows in this case by a theorem of Baumslag and Pride \cite{BP}.

A similar argument holds if $e$ does not disconnect $X$. Here $\G\cong A*\Z$, $A$ is residually free (as, of course, is $\Z$), and $R\in A*\Z$ is not conjugate to an element of $A$.  The result again follows from \cite{BP}.
\end{proof}

\section{Exponent sum zero}\label{es0}

Some versions of the Freiheitssatz for free products with amalgamation
where the relator is a proper power or satisfies some small-cancellation
condition have been proved in \cite{FRR,Juh1,Juh3}.  Similar results
for HNN extensions appear in \cite{Juh2}.

In \cite[Proposition 3.10]{How}, a version of the Freiheitssatz was proved in the case where 
$\G$ is a surface group and the exponent-sum of $e$ in $R$ is zero.  Here we extend this result to the general case of a $1$-acylindrical graph of groups
decomposition of a limit group.

\begin{Theorem}\label{expsumzero}
Let $\G$ be a limit group expressed as the fundamental group of a $1$-acylindrical graph of groups $(\mathcal{G},X)$, let $e$ be an edge of $X$ with cyclic edge group, and let $R\in\G$ be an element such that $\ell_e(R)>0=\sigma_e(R)$.  Let $A=\pi_1(\mathcal{G},Y)$, where $Y$ is one component of $X\sm\{e\}$.  Then the natural homomorphism $A\to\G/N(R)$ is injective.
\end{Theorem}

\begin{proof}
If the edge group $\G_e$ of $e$ is trivial, then the result follows from 
Proposition \ref{freeproduct}, so we may suppose that $\G_e$ is infinite
cyclic.  We argue by induction on $\ell_e(R)$.

\medskip
For the initial case $\ell_e(R)=2$ of the induction, we distinguish two cases.
First suppose that $e$ separates $X$, so that $X\sm\{e\}=Y\sqcup Z$. Then $\G=A\ast_C B$ where $A=\pi_1(\mathcal{G},Y)$ and $B=\pi_1(\mathcal{G},Z)$
are limit groups, and $C=\G_e$ is cyclic.   Moreover, up to conjugacy, we have $R=uv$, with $u\in A\sm C$ and $v\in B\sm C$. Let $a\in A$ and $b\in B$ be generators of $C$, such that $a=b$ in $\G$.

Since $v\notin C$ and the splitting is $1$-acylindrical, $b$
and $v$ generate a nonabelian subgroup of the limit group $B$.  But $2$-generator subgroups of limit groups are free or abelian, so the subgroup $F_2$ of $B$ generated by $b,v$ is free on these two elements.  Similarly, $a,u$ freely generate a free subgroup $F_1$ of $A$.   Let $F$ be the quotient of $F_1\ast_C F_2$ by $R=uv$.  Then $F=\<u,a,b,v| uv=1,a=b\>$ is free of rank $2$, and each of the natural maps $F_i\to F$ ($i=1,2$) is an isomorphism.  On the other hand, $$\G/N(R)=A\ast_{F_1} F \ast_{F_2} B=A\ast_F B,$$ and it follows that the natural maps $A,B\to\G/N(R)$ are injective.

\medskip
Next, suppose that $e$ does not separate $X$, so that $Y=X\sm\{e\}$, and $\G$ is an HNN extension of the form $\< A, t|ta=bt\>$.  Moreover, up to conjugacy,
$R=utvt^{-1}$, where $u\in A\sm\<a\>$ and $v\in A\sm\<b\>$.  

In particular, $R$ is contained in the subgroup $H$ of $\G$ generated by $A$ and $B=tAt^{-1}$.  Moreover, $H\cong A\ast_C B$ with $C=\<b\>=\<tat^{-1}\>$, and $R$ is cyclically reduced of length $2$ in $A\ast_C B$.  By the separating case, we see that the maps $A\to H/N(R)$, $B\to H/N(R)$ are injective.  Finally, $\G/N(R)$ is the HNN extension $\<H/N(R),t|tAt^{-1}=B\>$ of $H/N(R)$, so the natural map $A\to H/N(R)\to\G/N(R)$ is injective, as claimed.

\medskip
The inductive step is not essentially more complex than the initial case. Again we distinguish the separating and non-separating cases.

If $e$ is separating, then $\G=A\ast_C B$ where $A,B$ are limit groups and
$C=\<a\>=\<b\>$ is infinite cyclic.  Up to conjugacy, we have $R=u_1v_1\cdots u_kv_k$, with $u_i\in A\sm C$ and $v_i\in B\sm C$ for each $i$.  We may replace $A$ by the subgroup $A_0$ generated by $a,u_1,\dots,u_k$, and $B$ by the subgroup generated by $b,v_1,\dots,v_k$, for if the result holds in this case then $$\G/N(R)=A\ast_{A_0}(A_0\ast_C B_0)/N(R)\ast_{B_0} B,$$ and the general result follows.

Hence, without loss of generality, $B$ is generated by $b,v_1,\dots,v_k$.
Suppose that there is no epimorphism $\phi:B\to\Z$ with $\phi(b)=0=\phi(v_1\cdots v_k)$.  Nevertheless, there is an epimorphism
$\psi:B\to F$ from $B$ onto a nonabelian free group $F$, so the nonexistence of $\phi$ means that $F$ has rank $2$ and that $\psi(b)$ and $\psi(v_1\cdots v_k)$
are linearly independent modulo the commutator subgroup $[F,F]$. Let $\{x,y\}$ be a basis for $F$.  The map $\psi$ extends to a homomorphism $\psi:A\ast_C B\to A\ast_C F$ which is the identity on $A$ and $\psi$ on $B$.  The set of equations $\psi(R)=1,\psi(b)=a$ in $x,y$ over $A$ is nonsingular, since $\psi(b)$ and $\psi(v_1\cdots v_k)$ are linearly independent modulo $[F,F]$.  Moreover, $A$ is locally indicable.  It then follows from \cite{nonsing} that the natural map $A\to (A\ast_C F)/N(\psi(R))$ is injective. Since this map factors through the natural map $A\to (A\ast_C B)/N(R)$, the latter is also injective, as required.

\medskip
Hence we may assume that we have an epimorphism $\phi:B\to\Z$ with $\phi(b)=0=\phi(v_1\cdots v_k)$.  We may extend this to an epimorphism $\phi:A\ast_C B\to\Z$ by defining $\phi(\a)=0$ for all $\a\in A$.  The kernel $K$ of this homomorphism can be expressed as an infinite amalgamated product of $K_B=K\cap B$ with $A_n=t^n At^{-n}$ for $n\in\Z$ (where $t\in B$ is a choice of element with $\phi(t)=1$).  Here the cyclic subgroup $C_n=\<t^n at^{-n}\><A_n$ is amalgamated with $\<t^n bt^{-n}\><K_B$, for each $n$.  Up to conjugacy, we may rewrite $R$ as a word in $H:=K_B\ast_{C_0} A_0 \cdots \ast_{C_m} A_m$ for some $m\ge 0$ (which we may assume chosen to be least possible).  Analysis of this rewriting of $R$ shows that $m=0$ only if $\phi(v_i)=0$ for each $i$, which would contradict the combination of hypotheses that $b,v_1,\dots v_k$ generate
$B$, that $\phi(b)=0$, and that $\phi(B)=\Z$.   Hence $m>0$. Define $H_0:=K_B\ast_{C_0} A_0 \cdots \ast_{C_{m-1}} A_{m-1}$ and $H_1:=K_B\ast_{C_1} A_1 \cdots \ast_{C_m} A_m$. Again, analysis of the rewrite of $R$ shows that, as a word in $H=H_0\ast_{C_m} A_m$, it has length less than $2k=\ell_e(R)$, so that the natural map $H_0\to H/N(R)$ is injective by inductive hypothesis. Similarly, $H_1\to H/N(R)$ is injective.  Finally, conjugation by $t$ within $\G$ induces an isomorphism between $H_0$ and $H_1$, and $\G/N(R)$ can be expressed as an HNN extension $\<H,t|tH_0t^{-1}=H_1\>$. It follows that the natural map $A=A_0\to H_0\to\G/N(R)$ is injective, as required. 

 This completes the inductive step in the separating case.

\medskip
Finally, suppose that $e$ is nonseparating, so that $\G=\<A,t|ta=bt\>$. As in the initial step, $R$ belongs to the normal closure of $A$ in $\G$. In general, however, it will not belong to the subgroup generated by $A_0=A$ and $A_1=tAt^{-1}$.  However, up to conjugacy, it belongs to the subgroup $H$ generated by $A_n:=t^nAt^{-n}$ for $0\le n\le m$, for some $m\ge 1$.  As above, we make all choices so that $m$ is least possible, and define $H_0,H_1$ to be, respectively, the subgroups generated by $A_n$ for $0\le n\le m-1$ and by $A_n$ for $1\le n\le m$. 

If $m>1$, then the length of $R$, as a word in $H=H_0\ast_{C} A_m$ is strictly less than $\ell_e(R)$, so the inductive hypothesis tells us that $H_0\to H/N(R)$ is injective.  If $m=1$, then the length of $R$ in $A_0\ast_C A_1$ is equal to $\ell_e(R)$, so we cannot apply the inductive hypothesis.  However, we can apply the separating case which we have dealt with above.  Hence in all cases $H_0\to H/N(R)$ is injective.

Similarly, $H_1\to H/N(R)$ is injective.  Finally, $\G/N(R)$ is an HNN extension of $H/N(R)$, with associated subgroups $H_0$ and $H_1$ and stable letter $t$, and so $A\to H_0\to \G/N(R)$ is injective, as required.

This completes the inductive step in the nonseparating case, and hence the
proof of the theorem.
\end{proof}

\section{Exponent sum one}\label{es1}

In this section we prove a version of the Freiheitssatz when the exponent-sum
of the missing edge in the defining relator is $\pm 1$, using techniques developed in Klyachko's proof of the Kervaire conjecture for torsion-free
groups \cite{Klya} in the formulation of Fenn and Rourke \cite{FR}.

In our proof we will need the following result, which seems of independent interest. 

\begin{Theorem}\label{freeprod}
Let $\G$ be a limit group, $\G=A\ast_CB$ a $1$-acylindrical splitting over a cyclic group $C$, and $\gamma\in\G\sm A$.  Then the subgroup of $\G$ generated by $A$ and $\gamma$ is a free product of $A$ and $\<\gamma\>$.
\end{Theorem}

\begin{proof}
Let $T$ be the Bass-Serre tree of the splitting. Then $C$ is the stabiliser of an edge $e$ of $T$, and $A,B$ are the stabilisers of vertices $u,v$ respectively
that are incident to $e$.  We argue by induction on the distance $d=d(u,\gamma(u))$ (in the edge-path metric on $T$).  Note that $d$ is even since $T$ has a $\G$-invariant bipartite structure, and that $d>0$ since $\gamma\notin A=\mathrm{Stab}(u)$.

Now the first edge of the geodesic segment $\rho=[u,\gamma(u)]$ is incident to $u$, so is $\alpha_1(e)$ for some $\alpha_1\in A$, while the last edge is incident to $\gamma(u)$ so is $\gamma\alpha_2(e)$ for some $\alpha_2\in A$.
Replacing $\gamma$ by $\alpha_1^{-1}\gamma\alpha_2^{-1}\in A\gamma A$, we may assume that $e,\gamma(e)$ are the first and last edges in $\rho$.

In the initial case of the induction, $d=2$ and the midpoint of $\rho$
is the vertex $v$ stabilised by $B$.  Moreover, since the first edge
$\gamma(e)$ of $\gamma(\rho)$ coincides with the second edge of $\rho$,
it follows that $\gamma$ fixes the midpoint $v$ of $\rho$, so $\gamma\in B$.
If $c$ is a generator for $C=\mathrm{Stab}(e)$, then $\gamma,c$ do not
commute by the $1$-acylindrical property.  Since $\G$ is a limit group,
$\gamma,c$ generate a nonabelian free subgroup $F$ of $B$.

Now consider the $2$-component forest obtained from $T$ by removing the
edge $e$.  Let $T_u$, $T_v$ be the components of this forest that contain
$u$, $v$ respectively.  Since every element of $A\sm C$ maps $T_v$ into
$T_u$, and every element of $F\sm C$ maps $T_u$ into $T_v$, the Ping-Pong
Lemma shows that the subgroup of $\G$ generated by $A\cup F$ is the
amalgamated free product $$A\ast_C F= A\ast_{\<c\>}\<c,\gamma\>=A\ast\<\gamma\>,$$ as required.

Suppose then that $d\ge 4$.  If $\gamma$ is elliptic, then it must fix
the midpoint $x$ of $\rho$.  Since $d(u,x)=d/s\ge 2$ and the $\G$-action
is $1$-acylindrical, no nontrivial element of $A$ fixes the last edge
$f$ of the geodesic segment $\tau=[u,x]$.  If some power $\gamma^k$ of $\gamma$
fixes $f$, then $\gamma^k$ also fixes $\gamma(f)$, which is the last edge of the segment $\gamma(\tau)=[\gamma(u),x]$.  In particular $\gamma(f)\ne f$.  Since the action is $1$-acylindrical, and $\gamma^k$ fixes two distinct edges, then $\gamma^k=1$ so $k=0$.  Hence no nontrivial power of $\gamma$ fixes $f$.  In this situation we may again apply the Ping-Pong Lemma to the two components of $T\sm\{f\}$: nontrivial elements of $A$ map the component containing $x$ into the component containing $u$, while nontrivial elements of $\<\gamma\>$ map the component containing $u$ into the component containing $x$.  Hence $A\cup\<\gamma\>$ generates a free product $$A\ast\<\gamma\>\subset\G,$$ as required.

We may therefore assume that $\gamma$ is hyperbolic, with axis $L_\gamma$,
say.  In this case the segment $\rho$ meets $L_\gamma$ in a central
subsegment of length at most $d-2$.  (The edges $e,\gamma(e)$ of $\rho$
cannot lie on $L_\gamma$ since they both lie on a geodesic from $u$ to $\gamma(u)$.)  If the shortest path $P$ from $u$ to $L_\gamma$ contains an edge $f\ne e$, then we may apply the Ping-Pong Lemma to the components of $T\sm\{f\}$ as above to see that $A$ and $\<\gamma\>$ generate their free product in $\G$.

Finally, suppose that $\gamma$ is hyperbolic, with $v\in L_\gamma$ and $u\notin L_\gamma$.  If $c$ is a generator of $C$, then $c$ and $\gamma$ do not commute, by the $1$-acylindrical property.  Hence they
freely generate a nonabelian free subgroup $F$ of $\G$.  

We may assume that $L_\gamma\cap c^k(L_\gamma)=\{v\}$ for all $k\ne 0$.
For $c^k$ fixes $e\notin L_\gamma$, so by the $1$-acylindrical property
$c^k$ does not fix any edge
of $L_\gamma$.  Hence if $L_\gamma$ and $c^k(L_\gamma)$ share an edge,
then there are distinct edges $f,c^k(f)$ of $L_\gamma$ which are
incident at $v$.  It follows that one of $\gamma^{\pm 1}c^k$ is either
elliptic or has translation length less than that of $\gamma$, with 
neither $u$ nor $v$ contained in its axis.  Applying the above argument
shows that the subgroup generated by $A$ and $\gamma^{\pm 1}c^k$ is 
$$A*\<\gamma^{\pm 1} c^k\>=A*\<\gamma\>.$$

Let $\rho_0\subset L_\gamma$ denote the geodesic segment $[v,\gamma(v)]$.  Then the union of the translates $\phi(\rho_0)$ of $\rho_0$ for $\phi\in F$ form an $F$-invariant subtree 
$T_F$ of $T$.  We claim that $e\notin T_F$.  From this the result will again
follow from a Ping-Pong argument, for $e$ separates $T$ into two components,
$T_u\ni u$ and $T_v\supset T_F$.  Elements of $A\sm C$ map $T_v$ into $T_u$,
while elements of $F\sm C$ map $T_u$ into $T_v$, so the subgroup of $\G$
generated by $A$ and $F$ is $$A\ast_CF=A\ast_{\<c\>}\<c,\gamma\>=A\ast\<\gamma\>.$$

To prove the claim, note that the hypothesis that $L_\gamma\cap c^k(L_\gamma)=\{v\}$ for all
$k\ne 0$ means that the segments $\{c^k(\rho_0);k\in\Z\}$ and
$\{c^k\gamma^{-1}(\rho_0);k\in\Z\}$ are pairwise disjoint except for
their shared endpoint $v$.  
Orient $\rho_0$ from $v$ to $\gamma(v)$, and let $\rho_0^{-1}$ denote
$\gamma^{-1}(\rho_0)$ oriented from $v$ to $\gamma^{-1}(v)$.

It follows that, for a reduced word 
$$\phi=c^{k_0}\gamma^{\varepsilon_1}c^{k_1}\cdots\gamma^{\varepsilon_n}c^{k_n}$$
in $\{c,\gamma\}$, the concatenation of the $n$ paths
$$c^{k_0}\gamma^{\varepsilon_1}c^{k_1}\cdots\gamma^{\varepsilon_j}c^{k_j}(\rho_0^{\varepsilon_{j+1}}),~j=0,\dots,n-1,$$
is a reduced path from $v$ to $\phi(v)$.  In particular $d(v,\phi(v))=nd(v,\gamma(v))$, so we can have $e\in\phi(\rho_0)$ only
if $\phi\in\{c^k,c^k\gamma^{-1}\}$ for some $k\in\Z$.
But 
$$e\in c^k(\rho_0)\cup c^k\gamma^{-1}(\rho_0)\subset c^k(L_\gamma)\Rightarrow e=c^{-k}(e)\in L_\gamma,$$
a contradiction.

\end{proof}

\begin{Theorem} \label{expsumone}
Let $\G$ be a limit group expressed as the fundamental group of a $1$-acylindrical graph of groups $(\mathcal{G},X)$, let $e$ be an edge of $X$ with cyclic edge group, and let $R\in\G$ be an element such that $\ell_e(R)>1=\sigma_e(R)$.  Let $A=\pi_1(\mathcal{G},Y)$, where $Y=X\sm\{e\}$.  Then the natural homomorphism $A\to\G/N(R)$ is injective.
\end{Theorem}

\begin{proof}
Since $\sigma_e(R)\ne 0$, the edge $e$ cannot be separating, so $\G$ is an HNN extension $\<A,t|ta=bt\>$ and $A<\G$ is a limit group.

Now $\sigma_e(Rt^{-1})=0$, so $Rt^{-1}$ belongs to the normal closure of $A$ in $\G$, which is generated by $A_n:=t^nAt^{-n}$ for $t\in\Z$.  Conjugating by a suitable power of $t$, we may assume that $Rt^{-1}$ belongs to the subgroup $H_m$ of $\G$ generated by $A_n$ for $0\le n\le m$, for some $m\ge 0$, and that moreover this integer $m$ is the least possible that arises from all conjugates of $R$.

Note that $m>0$, since otherwise $R$ is conjugate to $\alpha t$ for some
$\alpha\in A$, and then $\ell_e(R)=1$, contrary to hypothesis.

Define $H_{-1}:=\<a\>$, and $H_k':=tH_kt^{-1}$ for all $k\ge -1$. Then the minimality assumption means that $R\in H_m\sm H_{m-1}$, and similarly $R\in H_m\sm H_{m-1}'$.

Now $H_m$ is an amalgamated free product: $$H_m = H_{m-1} \ast_{H_{m-2}'} H_{m-1}'.$$

Write $Rt^{-1}$ in reduced form with respect to this amalgamated free product decomposition:
$$Rt^{-1}=\beta_0\gamma_1\beta_1\cdots \gamma_k\beta_k,$$
with $\beta_i\in H_{m-1}$, $\gamma_i\in H'_{m-1}\sm H'_{m-2}$, and
$\beta_1,\dots,\beta_{k-1}\notin H'_{m-2}$.
Then replace each letter $\gamma_i$ from $H_{m-1}'\sm H_{m-2}'$ by $t\alpha_it^{-1}$, where $\alpha_i=t^{-1}\gamma_it\in H_{m-1} \sm H_{m-2}$.
This allows us to write $R$ as a word 
$$R = (\beta_0 t \alpha_1 t^{-1}) \cdots (\beta_{k-1} t \alpha_k t^{-1}) \beta_k t$$ with $k\ge 0$, $\beta_i\in H_{m-1}$ for each $i$, $\alpha_i\in H_{m-1}\sm H_{m-2}$ for each $i$, and $\beta_1,\dots,\beta_{k-1}\notin H'_{m-2}$.
Among all such expressions, for all conjugates of $R$, choose one that
minimises $k$ (subject to our existing hypothesis on the minimality of $m$).  Then
\begin{enumerate}
\item $k\ge 1$, for if $k=0$ we have $Rt^{-1}=\beta_0\in H_{m-1}$ which
contradicts the minimality of $m$;
\item $\beta_k\notin H'_{m-2}$, for if $\beta_k\in H'_{m-2}$ then $t^{-1}\beta_kt\in H_{m-2}\subset H_{m-1}$, and so the expression
$$gRg^{-1}=(w t \alpha_1 t^{-1}) \cdots (\beta_{k-2} t \alpha_{k-1} t^{-1}) \beta_{k-1} t$$ with $g=\alpha_kt^{-1}\beta_kt\in H_{m-1}$ and $w=g\beta_0\in H_{m-1}$ contradicts the minimality of $k$.
\end{enumerate}

Now put $\gamma=\beta_0$, and replace $R$ by its cyclic conjugate
$$R'=(\alpha_1t^{-1}\beta_1t)\cdots (\alpha_kt^{-1}\beta_kt)\gamma t,$$
with $k\ge 1$, $\alpha_i\in H_{m-1}\sm H_{m-2}$ for each $i$,
$\beta_i\in H_{m-1}\sm H'_{m-2}$ for each $i$, and $\gamma\in H_{m-1}$.

By Theorem \ref{freeprod}, the subgroup of $H_{m-1}$ generated by $H_{m-2}$ and $\alpha_i$ is a free product $H_{m-2}\ast\<\alpha_i\>$ for each $i$, and similarly the subgroup generated by $H_{m-2}'$ and $\beta_i$ is a free product $H_{m-2}'\ast\<\beta_i\>$. Therefore by a result of Fenn and Rourke {\rm\cite[Theorem 4.1]{FR}} (see also {\rm\cite[Lemma 2]{Klya}}), the system of equations
\begin{eqnarray*}
R'=\left(\prod_{i=1}^{k}\alpha_it^{-1}\beta_i t\right)\gamma t &=& 1\\
h^{\phi} &=& tht^{-1}~~(~\forall~h\in H_{m-2})
\end{eqnarray*}
has a solution over $H_{m-1}$, where $\phi$ is the isomorphism $H_{m-2}\to H_{m-2}'$ induced by conjugation by $t$. This is equivalent to saying that the natural map $H_{m-1} \to \< H_{m-1}, t | t H_{m-2} t^{-1} = H_{m-2}', R=1\>$ is injective.

But, since $\G$ can be expressed as the HNN extension 
$$\< H_{m-1}, t | t H_{m-2} t^{-1} = H_{m-2}' \>,$$ this shows that the natural map $$A\to H_{m-1} \to \G/N(R)$$ is injective, as required.
\end{proof}

\section{Two edges}\label{2e}

In this section we prove the following.

\begin{Theorem}\label{expsumtwo}
Let $\G$ be a limit group expressed as the fundamental group of a $1$-acylindrical graph of groups $(\mathcal{G},X)$, let $e,f$ be edges of $X$ with cyclic edge groups, and let $R\in\G$ be an element such that $\ell_e(R)\ge 1\le\ell_f(R)$.  Let $A=\pi_1(\mathcal{G},Y)$, where $Y$ is one component of $X\sm\{e,f\}$.  Then the natural homomorphism $A\to\G/N(R)$ is injective.
\end{Theorem}

\begin{proof}
We may assume that $\sigma_e(R)\ne 0\ne\sigma_f(R)$, for otherwise we can apply Theorem \ref{expsumzero}.  In particular, neither $e$ nor $f$ separates $X$.  It is, however, possible that $\{e,f\}$ separates $X$ as a pair of edges. 

We argue by induction on $\ell_e(R)+\ell_f(R)\ge 2$.  In the initial case,
$\ell_e(R)=\ell_f(R)=1$.  If $\{e,f\}$ separates $X$, then we have a presentation for $\G$ of the form $(A*B,f|a_e=b_e,a_ff=fb_f)$, where
$a_e\in A,b_e\in B$ generate the edge group of $e$ and $a_f\in A,b_f\in B$ generate the edge group of $f$.  We can also write $R$ (up to cyclic conjugation and inversion) in the form $xfy$ with $x\in A$ and $y\in B$.  Hence $\G/N(R)$ has a presentation of the form $(A*B|a_e=b_e,xa_fx^{-1}=y^{-1}b_fy)$. Now by $1$-acylindricity, the subgroup of $A$ generated by $\{a_e,xa_fx^{-1}\}$ is nonabelian, and hence is free of rank $2$. Similarly the subgroup of $B$ generated by $\{b_e,y^{-1}b_fy\}$ is free of rank $2$.  Then $\G/N(R)$ is just the free product of $A$ and $B$ amalgamating these two subgroups of rank $2$.  In particular, the natural map $A\to\G/N(R)$ is injective.

A similar argument applies if the pair of edges $\{e,f\}$ does not separate $X$.  Here $\G=(A,e,f|a_ee=eb_e,a_ff=fb_f)$, $R=xfye^{-1}$, and $\G/N(R)=(A*B|a_ee=eb_e,xa_fx^{-1}e=ey^{-1}b_fy)$ is an HNN extension of $A$ with associated subgroups free of rank $2$.  Again the natural map $A\to\G/N(R)$ is injective.

\medskip
The inductive step of the proof proceeds in a similar fashion to that in Theorem \ref{expsumzero}.  Suppose first that $\{e,f\}$ separates $X$.  Then, with 
suitable choices of orientation, the edges $\{e,f\}$ appear in $R$ with alternating signs.  Since $\sigma_e(R)=-\sigma_f(R)\ne 0$, there must be a (cyclic) subword of $R$ of the form $e^{-1}a_0f$ or $f^{-1}a_0^{-1}e$ with $a_0\in A$, and another (cyclic) subword $(eb_0f^{-1})^{\pm 1}$ with $b_0\in B$. Replacing each occurrence of the symbol $f$ in $R$ by $a_0^{-1}fb_0$ does not change the problem, but allows us to assume that $R$ has subwords of the form $(e^{-1}f)^{\pm 1}$ and $(ef^{-1})^{\pm 1}$.

By the usual trick, we may assume that $A$ is generated by $a_e,a_f$ and the $A$-letters that occur in $R$, while a similar property holds for $B$.  (Letting $A_0,B_0$ be the subgroups generated by these letters, and $\G_0$ the corresponding subgroup of $\G$, if the theorem is true for $\G_0$ then $\G/N(R)=A*_{A_0}\G_0/N(R)*_{B_0}B$ and the result follows for $\G$.)

Since the graph of groups is $1$-acylindrical, $a_e$ and $a_f$ do not commute. Hence $A$ is a nonabelian limit group, and so has a nonabelian free homomorphic image.  It follows that $Hom(A,\Q)$ has $\Q$-dimension at least $2$.  A similar property holds for $B$, and so $Hom(\G,\Q)$ has $\Q$-dimension at least $5$.  In particular, we may find an epimorphism $\phi:\G\to\Z$ with $a_e,a_f,R\in K:=\mathrm{Ker}(\phi)$.  In the induced graph-of-groups decomposition  $K=\pi_1(\mathcal{K},\bar{X})$ for $K$, the lifts of the edges $e,f$ are indexed by integers, and the various choices have been made to ensure that at least three distinct edges from $\{e_n,f_n;n\in\Z\}$ occur in any rewrite of $R\in K$. Without loss of generality, there is a rewrite $R_0$ such that the minimum index of an $e$-edge involved in $R_0$ is $0$, and the minimum index of an $f$-edge is also $0$.  If $p,q$ denote the maximum index of $e$-edges, $f$-edges respectively that occur in $R_0$, then $p+q\ge 1$. Let $X_0$ be the subgraph of the covering graph $\bar{X}$ that contains all vertices, all edges that are not $e$-edges or $f$-edges, and the $e$-edges $e_0,\dots,e_{p-1}$ and $f$-edges $f_0,\dots,f_{q-1}$. Let $X_1$ denote the subgraph of $\bar{X}$ that contains
all vertices, all edges that are not $e$-edges or $f$-edges, and the $e$-edges $e_1,\dots,e_p$ and $f$-edges $f_1,\dots,f_q$. Finally, let $X'$ denote the subgraph of $\bar{X}$ that contains all vertices, all edges that are not $e$-edges or $f$-edges, and the $e$-edges $e_0,\dots,e_p$ and $f$-edges $f_0,\dots,f_q$. Then the inductive hypothesis ensures that the natural maps from $A':=\pi_1(\mathcal{K},X_0)$ and from $B':=\pi_1(\mathcal{K},X_1)$ to $C':=\pi_1(\mathcal{K},X')/N(R_0)$ are injective.  Finally, $K/N(\{R_n;n\in\Z\})$ is an HNN-extension of $C'$ with associated subgroups $A'$ and $B'$, so it follows that $A'\to C'$ and $B'\to C'$ are injective.  Hence it also follows that $A\to \G/N(R)$ is injective, as required.

\medskip
Finally, if $\{e,f\}$ do not separate $X$, then $\pi_1(X)$ is free of rank at least $2$.  Indeed, we can find an epimorphism $\phi:\G\to\Z$ that vanishes on $\pi_1(\mathcal{G},X\sm\{e,f\})$ and on $R$, and analyse the kernel $K$ of $\phi$ as in the separating case above.  This time, $\bar X$ consists of a number of copies of $X\sm\{e,f\}$ indexed by $\Z$, connected by lifts $e_n,f_n$ of the edges $e,f$.  When we rewrite $R$ and try to apply the inductive
hypothesis, we find that we cannot in general ensure that we have reduced $\ell_e(R)+\ell_f(R)$.  In other words, it is possible that $p=0=q$ in the above analysis.  However, in this case the pair $\{e_0,f_0\}$ separates two copies of $X\sm\{e,f\}$, so we can instead apply the separating case which we have just proved.

Again, we can express $\G/N(R)$ as an HNN-extension, where the injectivity of the associated subgroups is guaranteed either by the inductive hypothesis or by the separating case.

This completes the inductive step, and hence the proof.
\end{proof}

\end{document}